\documentclass[pdflatex,sn-mathphys-num]{sn-jnl}
\usepackage{amsmath,amssymb,amsfonts,amsthm}
\theoremstyle{thmstyleone}
\newtheorem{theorem}{Theorem}
\newtheorem{lemma}[theorem]{Lemma}
\newtheorem{proposition}[theorem]{Proposition}

\theoremstyle{thmstyletwo}
\newtheorem{remark}[theorem]{Remark}
\newtheorem{example}[theorem]{Example}
\newcommand{\CNJ}{C_{\mathrm{NJ}}}
\newcommand{\JB}{J_{\perp}}
\newcommand{\conv}{\operatorname{conv}}
\newcommand{\R}{\mathbb R}
\newcommand{\T}{\mathsf T}
\begin{document}
\title[Orthogonal James and von Neumann--Jordan constants]{An Improved Inequality between the von Neumann–Jordan and James Constants}
\author[1]{\fnm{Zhiyao} \sur{Fang}}
\email{fzy73@stu.aqnu.edu.cn}

\author*[1]{\fnm{Qi} \sur{Liu}}
\email{liuq67@aqnu.edu.cn}

\affil[1]{%
	\orgdiv{School of Mathematics and Statistics},
	\orgname{Anqing Normal University},
	\orgaddress{
		\city{Anqing},
		\postcode{246133},
		\state{Anhui},
		\country{China}
	}
}

\abstract{
	We provide a strengthened answer to the questions posed by Alonso, Mart\'in, and Papini [Studia Math. 188 (2008), 135--150] concerning the comparison between the von Neumann--Jordan and James constants . The classical inequality $\CNJ(X)\leq J(X)$ between the von Neumann--Jordan and James constants was established by Takahashi and Kato [J. Math. Anal. Appl. 359 (2009), 602--609], Wang [Proc. Amer. Math. Soc. 138 (2010), 695--701], and Yang and Li [Appl. Math. Lett. 23 (2010), 277--281], while related sharp estimates were obtained by Kato and Takahashi [Rend. Circ. Mat. Palermo (2) Suppl. 82 (2010), 75--91]. 
	We sharpen this inequality to $\CNJ(X)\leq\JB(X)$,
	where $\JB(X)$ is the Birkhoff--James orthogonal James
	constant introduced by Baronti and Papini
	[Constr. Math. Anal. 5 (2022), 37--45].
	 Our proof employs optimal distance ellipses and alternating contact directions.
}

\maketitle

\section{Introduction}\label{sec:introduction}

Let $X$ be a real Banach space with $\dim X\geq2$, and let $S_X$ denote its unit sphere. The von Neumann--Jordan constant is

\begin{equation}\label{eq:nj-definition}
	\CNJ(X)=\sup_{(x,y)\ne(0,0)}
	\frac{\|x+y\|^2+\|x-y\|^2}{2(\|x\|^2+\|y\|^2)}.
\end{equation}
Introduced by Clarkson \cite{clarkson1937},
this constant measures the largest deviation
from the parallelogram identity.
The classical James constant, introduced by
Gao and Lau \cite{gaolau1990}, is defined by
\[
J(X)=\sup_{x,y\in S_X}
\min\{\|x+y\|,\|x-y\|\}.
\]

measures how large both diagonals generated by two unit vectors can be.
Maligranda conjectured that
\[
\CNJ(X)\leq 1+\frac{J(X)^2}{4}
\]
(see \cite{alonso2008,saejung2006}).
Alonso, Mart\'in, and Papini \cite{alonso2008}
subsequently established a stronger estimate
and raised the questions of whether
$\CNJ(X)\leq J(X)$ holds for every Banach space
and whether equality occurs only when both constants
equal $2$.
The inequality $\CNJ(X)\leq J(X)$ was later established
by Takahashi and Kato \cite{takahashikato2009},
Wang \cite{wang2010}, and Yang and Li \cite{yangli2010},
with the equality case also characterized.
Further sharp estimates concerning these constants
were obtained by Kato and Takahashi
\cite{kato2010sharp}.

Recall that $x\in X$ is Birkhoff--James orthogonal to $y\in X$, written $x\perp_B y$, if
\begin{equation}\label{eq:bj-definition}
 \|x+t y\|\geq\|x\|\qquad(t\in\R).
\end{equation}
This relation originates in the work of Birkhoff and James \cite{birkhoff1935,james1947}. It need not be symmetric. Baronti and Papini \cite{baronti2022} introduced the orthogonal James constant
\begin{equation}\label{eq:orthogonal-james-definition}
 \JB(X)=\sup_{\substack{x,y\in S_X\\x\perp_B y}}
 \min\{\|x+y\|,\|x-y\|\}.
\end{equation}
Since $\JB(X)\leq J(X)$, replacing $J(X)$ by $\JB(X)$ in the von Neumann--Jordan comparison gives a stronger assertion. Orthogonal and unrestricted constants can differ, and orthogonality does not require the two diagonal lengths to agree. Consequently, the unrestricted comparison does not itself establish the orthogonal one.

To bridge this gap, we adopt a geometric approach based on
the two-dimensional structure of normed spaces. Rather than
attempting to impose Birkhoff--James orthogonality on vectors
appearing in the von Neumann--Jordan quotient, we exploit
the contact geometry of optimal distance ellipses to
construct suitable orthogonal unit vectors. This approach
establishes a quantitative comparison through the
Banach--Mazur distance and ultimately yields the desired
inequality for general real Banach spaces.

Throughout this paper, a \emph{real normed plane} means
a two-dimensional normed linear space over $\mathbb{R}$.
For real normed planes $E$ and $F$, their Banach--Mazur distance is
\[
d(E,F)=\inf\{\|T\|\,\|T^{-1}\|:T:E\longrightarrow F
\text{ is a linear isomorphism}\}.
\]
We write $\ell_2^2$ for the Euclidean plane. Our main result is the following.

\begin{theorem}\label{thm:main}
Every real normed plane $E$ satisfies
\begin{equation}\label{eq:planar-chain}
 \CNJ(E)\leq d(E,\ell_2^2)^2\leq\JB(E).
\end{equation}
Consequently, every real Banach space $X$ with $\dim X\geq2$ satisfies
\begin{equation}\label{eq:banach-comparison}
 \CNJ(X)\leq\JB(X).
\end{equation}
The coefficient one in both inequalities of \eqref{eq:planar-chain}, and in \eqref{eq:banach-comparison}, is sharp.
\end{theorem}

The first inequality in \eqref{eq:planar-chain} follows directly from comparison with a Euclidean norm. The second is the essential step. We place the unit ball between an optimal pair of concentric, homothetic ellipses. The contact structure belongs to the theory of distance ellipsoids initiated by Ader and revisited by Grundbacher and Kobos \cite{grundbacher2026}; the planar alternation is also treated by Babev, Ivanov, and Nikolov \cite{babev2026}. We include the necessary contact argument to keep the proof self-contained. Its role here is to provide an orthogonal unit pair for which both diagonal lengths are at least the squared ellipse ratio.

Section~\ref{sec:ellipses} develops the geometry of optimal
distance ellipses and constructs the required
Birkhoff--James orthogonal pair.
Section~\ref{sec:comparison} establishes the main comparison
and discusses sharpness and equality cases.

\section{Geometric Constructions in Normed Planes}\label{sec:ellipses}

Let $K$ be the unit ball of a real normed plane, represented as an origin-symmetric convex body in $\R^2$. An optimal distance ellipse is an origin-centred ellipse $A$ for which
\[
 A\subset K\subset R A
\]
holds with the smallest possible $R$. This minimum is $d(E,\ell_2^2)$. Indeed, a linear isomorphism to the Euclidean plane gives such an inclusion after a scalar normalization, and the product of the two operator norms is precisely its outer-to-inner radius ratio.

For completeness, the minimum is attained. Fix a Euclidean disk $D$ and constants $0<m\leq M$ such that $mD\subset K\subset MD$. A positive definite symmetric matrix $Q$ defines the Euclidean norm $|z|_Q=(z^{\T}Qz)^{1/2}$. Write $a_Q,b_Q$ for its inner and outer radii on $K$, so that
\[
 \{z:|z|_Q\leq a_Q\}\subset K\subset\{z:|z|_Q\leq b_Q\}.
\]
Both radii are continuous in $Q$. The ratio is unchanged when $Q$ is multiplied by a positive scalar, so impose $\det Q=1$. Since
\[
 b_Q\geq m\sqrt{\lambda_{\max}(Q)},\qquad
 a_Q\leq M\sqrt{\lambda_{\min}(Q)},
\]
a bounded ratio confines $Q$ to a compact subset of the positive definite matrices with determinant one. A minimizing sequence therefore has a positive definite limit.

Choose linear coordinates and a scalar normalization so that the optimal ellipse is the Euclidean unit disk $D$. We then have
\begin{equation}\label{eq:optimal-annulus}
 D\subset K\subset R D,\qquad R=d(E,\ell_2^2).
\end{equation}
Let $\langle\cdot,\cdot\rangle$ and $|\cdot|$ denote the resulting Euclidean inner product and norm, and let $S^1=\partial D$. Define
\begin{equation}\label{eq:contact-sets}
 U=\{u\in S^1:Ru\in K\},\qquad
 V=\{v\in S^1:h_K(v)=1\},
\end{equation}
where $h_K(v)=\max_{z\in K}\langle v,z\rangle$ is the support function. The sets $U,V$ are nonempty, compact, and origin-symmetric. For $u\in U$ and $v\in V$, symmetry gives
\begin{equation}\label{eq:contact-separation}
 |\langle u,v\rangle|\leq R^{-1}.
\end{equation}
The points $v\in V$ are inner contact points as well as contact normals: $v\in D\subset K$ and $\langle v,v\rangle=h_K(v)=1$.

\begin{lemma}\label{lem:contact-hulls}
The convex hulls of the rank-one contact matrices intersect:
\begin{equation}\label{eq:contact-hulls}
 \conv\{uu^{\T}:u\in U\}\cap
 \conv\{vv^{\T}:v\in V\}\ne\varnothing.
\end{equation}
\end{lemma}

\begin{proof}
If the two compact convex sets were disjoint, strict separation in the space of symmetric matrices would give a symmetric matrix $H$ with
\begin{equation}\label{eq:strict-matrix-separation}
 \max_{u\in U}u^{\T}Hu<\min_{v\in V}v^{\T}Hv.
\end{equation}
For small $\varepsilon>0$, put $Q_\varepsilon=I+\varepsilon H$. Its outer and inner radii satisfy
\begin{align}
 b(\varepsilon)^2&=\max_{z\in K}z^{\T}Q_\varepsilon z,
 \label{eq:perturbed-outer}\\
 a(\varepsilon)^2&=\min_{v\in S^1}
 \frac{h_K(v)^2}{v^{\T}Q_\varepsilon^{-1}v}.
 \label{eq:perturbed-inner}
\end{align}
The second identity follows by comparing support functions: the ellipse $z^{\T}Q_\varepsilon z\leq a^2$ has support function $a\sqrt{v^{\T}Q_\varepsilon^{-1}v}$.

At $\varepsilon=0$, the active points in \eqref{eq:perturbed-outer} are $RU$, while those in \eqref{eq:perturbed-inner} are $V$. Therefore
\begin{equation}\label{eq:ratio-derivative}
 \left.\frac{d}{d\varepsilon}\log\frac{b(\varepsilon)^2}{a(\varepsilon)^2}\right|_{0+}
 =\max_{u\in U}u^{\T}Hu-\min_{v\in V}v^{\T}Hv<0.
\end{equation}
Here the elementary differentiation rule for a compact maximum applies because the first-order expansions are uniform. Explicitly, if $g_\varepsilon=g_0+\varepsilon g_1+o(\varepsilon)$ uniformly on a compact set, evaluating at a maximizer of $g_0$ gives the lower derivative bound. Maximizers of $g_\varepsilon$ have accumulation points among the maximizers of $g_0$, which gives the upper bound. The right derivative is thus the maximum of $g_1$ on the active set; the minimum rule follows by changing signs.

Equation~\eqref{eq:ratio-derivative} contradicts optimality in \eqref{eq:optimal-annulus}. The perturbation need not preserve $\det Q=1$, since the radius ratio is invariant under scalar multiplication of $Q$. This proves \eqref{eq:contact-hulls}.
\end{proof}

The following lemma establishes the alternating
configuration of contact directions and provides
the angular estimates needed for the subsequent
geometric construction.

\begin{lemma}\label{lem:alternating-directions}
Suppose $R>1$ and put $\alpha=\arccos(R^{-1})$. There exist four directions in cyclic order on $\R\mathbb P^1$,
\[
 u_1,\ v_1,\ u_2,\ v_2,\qquad u_i\in U,\quad v_i\in V.
\]
Their four consecutive angular gaps $\delta_i$ satisfy
\begin{equation}\label{eq:angular-gaps}
 \delta_i\geq\alpha,\qquad \sum_{i=1}^4\delta_i=\pi.
\end{equation}
In particular, $0<\alpha\leq\pi/4$. One of the two outer directions, denoted by $u$, has its two neighbouring inner directions at angular gaps $\delta_-,\delta_+$ satisfying
\begin{equation}\label{eq:good-neighbours}
 \alpha\leq\delta_-,\delta_+\leq\pi/2-\alpha.
\end{equation}
The acute angle $\theta$ between $u$ and the other outer direction satisfies
\begin{equation}\label{eq:other-outer-angle}
 2\alpha\leq\theta\leq\pi/2.
\end{equation}
\end{lemma}

\begin{proof}
The map $u\mapsto uu^{\T}$ identifies the projective circle with a circle in the affine plane of trace-one symmetric matrices. If $u=(\cos\varphi,\sin\varphi)$, its circle parameter is $2\varphi$. By \eqref{eq:contact-separation}, the two contact images are disjoint. By Lemma~\ref{lem:contact-hulls}, their convex hulls intersect.

Choose finite subsets witnessing this intersection, using the planar convex hull theorem. No point of the circle can lie in the convex hull of other distinct circle points. Hence neither of the resulting inscribed convex polygons can contain the other. Their boundaries must cross, and a pair of crossing chords gives four alternating endpoints. This argument includes a polygon that is a line segment; coincident chord edges would have the same endpoints and are excluded by disjointness. Pulling these four endpoints back to $\R\mathbb P^1$ gives the required directions.

For a consecutive gap $\delta_i\in(0,\pi)$, \eqref{eq:contact-separation} yields $|\cos\delta_i|\leq\cos\alpha$, in particular $\delta_i\geq\alpha$. This proves \eqref{eq:angular-gaps} and $\alpha\leq\pi/4$. At most one gap can be strictly larger than $\pi/2-\alpha$: two such gaps together with the remaining two gaps would have total length greater than $\pi$. Each gap is adjacent to exactly one of the two outer directions, so one outer direction has neither neighbouring gap too large. This proves \eqref{eq:good-neighbours}. Each of the two projective arcs between the outer directions contains two consecutive gaps and has length at least $2\alpha$. Their shorter length is $\theta$, proving \eqref{eq:other-outer-angle}.
\end{proof}

We now use the preceding geometric configuration
to construct a suitable Birkhoff--James orthogonal
pair. The following proposition is the key step
in establishing the comparison between the
Banach--Mazur distance and the orthogonal
James constant.

\begin{proposition}\label{prop:orthogonal-pair}
Under \eqref{eq:optimal-annulus}, there exist $x,y\in S_E$ such that $x\perp_B y$ and
\begin{equation}\label{eq:diagonal-bound}
 \min\{\|x+y\|,\|x-y\|\}\geq R^2.
\end{equation}
\end{proposition}

\begin{proof}
If $R=1$, the norm is Euclidean, and an orthonormal pair has both diagonal lengths $\sqrt2\geq R^2$. Suppose $R>1$, and choose $u$ as in Lemma~\ref{lem:alternating-directions}. Let $e$ be a Euclidean unit vector perpendicular to $u$ and set
\begin{equation}\label{eq:chosen-pair}
 x=Ru,\qquad y=re,\qquad r=\max\{s\geq0:se\in K\}.
\end{equation}
Then $x,y\in S_E$ and $1\leq r\leq R$. The linear functional $f(z)=\langle u,z\rangle/R$ has norm one because $K\subset RD$ and $f(x)=1$. Since $f(y)=0$, we have
\[
 \|x+t y\|\geq f(x+t y)=1\qquad(t\in\R).
\]
Thus $x\perp_B y$.

Represent the neighbouring inner contacts by
\[
 v_+=\cos\delta_+\,u+\sin\delta_+\,e,\qquad
 v_-=\cos\delta_-\,u-\sin\delta_-\,e.
\]
The functionals $z\mapsto\langle v_\pm,z\rangle$ have norm one. Therefore
\begin{align}
 \|x+y\|&\geq R\cos\delta_++r\sin\delta_+,
 \label{eq:plus-diagonal}\\
 \|x-y\|&\geq R\cos\delta_-+r\sin\delta_-.
 \label{eq:minus-diagonal}
\end{align}
The function $h(\delta)=R\cos\delta+r\sin\delta$ is concave on $[\alpha,\pi/2-\alpha]$. Moreover,
\[
 h(\alpha)-h(\pi/2-\alpha)
 =(R-r)(\cos\alpha-\sin\alpha)\geq0.
\]
Its minimum on this interval is consequently the right endpoint value, and
\begin{equation}\label{eq:common-lower-bound}
 h(\delta)\geq\tan\alpha+r\cos\alpha.
\end{equation}
We show that this lower bound is at least $R^2$.

Put $t=\tan\alpha\in(0,1]$, so that $R^2=1+t^2$. If $0<\alpha\leq\pi/6$, then $t\leq1/\sqrt3<2/3$. Using $r\geq1$ and the inequality $(1+s)^{-1/2}\geq1-s/2$ for $s\geq0$, we obtain
\begin{equation}\label{eq:small-angle-bound}
 \tan\alpha+r\cos\alpha
 \geq t+1-\frac{t^2}{2}\geq1+t^2=R^2.
\end{equation}

Suppose now that $\pi/6\leq\alpha\leq\pi/4$. We use the other outer contact to improve the radial bound on $r$. By origin symmetry and a possible change of sign of $e$, this contact is
\[
 p=R(\cos\theta\,u+\sin\theta\,e),
 \qquad 2\alpha\leq\theta\leq\pi/2.
\]
Changing the sign of $e$ only exchanges the two diagonal estimates; it does not assume any reflection symmetry of $K$. Put $\phi=\theta+\alpha$. The point
\[
 q=\cos\phi\,u+\sin\phi\,e
\]
belongs to $D$, and $\langle q,p\rangle=R\cos\alpha=1$. Hence the line through $p$ and $q$ is tangent to $D$ at $q$ and has equation
\[
 \cos\phi\,z_1+\sin\phi\,z_2=1
\]
in the basis $(u,e)$. Since
\[
 \pi/2\leq3\alpha\leq\phi\leq\pi/2+\alpha\leq3\pi/4,
\]
the first coordinate of $q$ is nonpositive and that of $p$ is nonnegative. The segment $[p,q]\subset K$ meets the positive $e$-axis at coordinate $1/\sin\phi$. It follows that
\begin{equation}\label{eq:radial-tangent-bound}
 r\geq\frac1{\sin(\theta+\alpha)}\geq\frac1{\sin3\alpha}.
\end{equation}
The last inequality uses the decrease of sine on $[\pi/2,\pi]$. If $\theta=\pi/2$, the intersection point is $p$ itself, so the same reasoning applies.

Using $\sin3\alpha=t(3-t^2)/(1+t^2)^{3/2}$ in \eqref{eq:radial-tangent-bound} gives
\begin{equation}\label{eq:large-angle-bound}
 \tan\alpha+r\cos\alpha
 \geq t+\frac{1+t^2}{t(3-t^2)}\geq1+t^2.
\end{equation}
Indeed,
\begin{equation}\label{eq:scalar-factorization}
 t+\frac{1+t^2}{t(3-t^2)}-(1+t^2)
 =\frac{(1-t)^2(t^3+t^2-t+1)}{t(3-t^2)}\geq0,
\end{equation}
because the denominator is positive and $t^3+t^2-t+1=t^3+t^2+(1-t)\geq0$ for $0<t\leq1$.

Combining \eqref{eq:plus-diagonal}, \eqref{eq:minus-diagonal}, and \eqref{eq:common-lower-bound} with \eqref{eq:small-angle-bound} or \eqref{eq:large-angle-bound} proves \eqref{eq:diagonal-bound}.
\end{proof}

\section{Main Results and Examples}\label{sec:comparison}

\begin{proof}[Proof of Theorem~\ref{thm:main}]
Let $E$ be a real normed plane. In the coordinates of \eqref{eq:optimal-annulus},
\begin{equation}\label{eq:norm-comparison}
 |z|/R\leq\|z\|\leq|z|\qquad(z\in E).
\end{equation}
The Euclidean parallelogram identity therefore gives
\[
 \frac{\|z+w\|^2+\|z-w\|^2}{2(\|z\|^2+\|w\|^2)}
 \leq R^2\frac{|z+w|^2+|z-w|^2}{2(|z|^2+|w|^2)}=R^2
\]
for $(z,w)\ne(0,0)$. Thus $\CNJ(E)\leq R^2$. Proposition~\ref{prop:orthogonal-pair} gives $R^2\leq\JB(E)$ and proves \eqref{eq:planar-chain}.

For an arbitrary real Banach space $X$ of dimension at least two, every pair of vectors lies in a two-dimensional subspace, enlarging a one-dimensional span if necessary. Hence
\begin{equation}\label{eq:nj-plane-reduction}
 \CNJ(X)=\sup_{\substack{E\subset X\\\dim E=2}}\CNJ(E).
\end{equation}
Birkhoff--James orthogonality of two vectors is unchanged on passing to a containing subspace: condition \eqref{eq:bj-definition} involves only vectors in their span. Nonzero orthogonal vectors are linearly independent, since $y=cx$ with $c\ne0$ would give $x+t y=0$ for $t=-1/c$. Consequently,
\begin{equation}\label{eq:j-plane-reduction}
 \JB(X)=\sup_{\substack{E\subset X\\\dim E=2}}\JB(E).
\end{equation}
Taking suprema in \eqref{eq:planar-chain} proves \eqref{eq:banach-comparison}.

\end{proof}

\begin{remark}
	\label{rem:euclidean-and-strict}
	The inequalities in Theorem~\ref{thm:main}
	are sharp. Indeed, for $E=\ell_\infty^2$,
	take $x=(1,1)$ and $y=(1,-1)$. Then
	$x,y\in S_E$, $x\perp_B y$, and
	\[
	\|x+y\|_\infty=\|x-y\|_\infty=2.
	\]
	Consequently, $\CNJ(E)=\JB(E)=2$, and
	\eqref{eq:planar-chain} yields
	$d(E,\ell_2^2)^2=2$.
	
	On the other hand, for the Euclidean plane,
	$\CNJ=1$, $d=1$, and $\JB=\sqrt2$.
	More generally, the estimates in
	Proposition~\ref{prop:orthogonal-pair} imply
	\[
	d(E,\ell_2^2)^2<\JB(E)
	\quad\text{if }d(E,\ell_2^2)<\sqrt2.
	\]
	Indeed, for $0<\alpha\leq\pi/6$, the
	difference in \eqref{eq:small-angle-bound}
	is $t-3t^2/2>0$, while for
	$\pi/6\leq\alpha<\pi/4$, the difference
	in \eqref{eq:scalar-factorization} is positive.
	Thus equality in the second inequality
	of \eqref{eq:planar-chain} can occur only
	at the maximal value $2$.
\end{remark}

\begin{example}\label{ex:strict-improvement}
	Consider $X=(\mathbb{R}^2,\|\cdot\|)$ with
	\[
	\|(a,b)\|=\max\left\{|a|,|b|,
	\left|a+\frac{b}{2}\right|\right\}.
	\]
	Baronti and Papini \cite{baronti2022} showed that
	\[
	J(X)=\frac74,\qquad \JB(X)=\frac53.
	\]
	Moreover, by convexity, the computation of $\CNJ(X)$
	reduces to pairs of vertices of the unit ball.
	Taking $u=(1/2,1)$ and $v=(-1,1)$ yields
	\[
	\begin{aligned}
		\CNJ(X)
		&=\max_{0\le t\le1}
		\frac{(1+t)^2+(1+t/2)^2}{2(1+t^2)}\\
		&=\frac{13+3\sqrt{17}}{16}.
	\end{aligned}
	\]
	Consequently,
	\[
		\CNJ(X)=\frac{13+3\sqrt{17}}{16}
		<\JB(X)=\frac53<J(X)=\frac74.
	\]
	Thus, Theorem~\ref{thm:main} provides a strict
	improvement over the classical inequality
	$\CNJ(X)\le J(X)$.
\end{example}

\begin{remark}\label{rem:orthogonality}
	The James constant admits the equivalent characterization
	\cite{papini2002}
	\[
	J(X)=\sup_{\substack{x,y\in S_X\\x\perp_I y}}
	\|x+y\|,
	\]
	where $x\perp_I y$ denotes isosceles orthogonality,
	defined by $\|x+y\|=\|x-y\|$.
	Thus the classical comparison between the
	von Neumann--Jordan and James constants
	can be interpreted in terms of isosceles
	orthogonality. Our result reveals that
	Birkhoff--James orthogonality, although
	generally distinct from isosceles orthogonality,
	also provides effective control of the
	von Neumann--Jordan constant, with potentially
	sharper estimates. This highlights the
	geometric significance of both orthogonality
	relations and their effectiveness in capturing
	the structure of Banach spaces.
\end{remark}

\section*{Acknowledgements}

Thanks to all the members of the Functional Analysis Research team of the College of Mathematics and Statistics of Anqing Normal University for their discussion and correction of the difficulties and errors encountered in this
paper. The authors cordially thank Professor Hai Zhang for his encouraging advice and the first author would like to acknowledge his strong support for undergraduate	students in carrying out basic mathematics research.

\section*{Declarations}

\subsection*{Competing Interests}
The authors declare that there are no competing interests.

\subsection*{Funding Information}
Not Applicable.

\subsection*{Author contribution}
These authors contributed equally to this work. All authors have read and agreed to the published version of the manuscript.

\subsection*{Data Availability Statement}
Not Applicable.


\end{document}